\documentclass{amsart}

\usepackage{amssymb, array, verbatim, amscd, color}
\usepackage{amsmath, amsfonts,amsthm}
\usepackage{enumerate}
\usepackage[all]{xy}
\usepackage{xy}

\theoremstyle{plain}

\newtheorem{theorem}{Theorem}
\newtheorem{corollary}[theorem]{Corollary}
\newtheorem{lemma}[theorem]{Lemma}
\newtheorem{proposition}[theorem]{Proposition}

\theoremstyle{definition}

\newtheorem{example}[theorem]{Example}

\newtheorem{remark}[theorem]{Remark}

\newcommand{\bbN}{\mathbb N}

\newcommand{\bbF}{\mathbb F}

\newcommand{\N}{{\mathbb N}}

\newcommand{\tr}[1]{\mathrm{trace}\left(#1\right)}

\newcommand{\rank}{\operatorname{rank}}

\newcommand{\Mat}{\operatorname{\mathbb{M}}}

\begin{document}

\title{Nil-clean companion matrices} 
\author{Simion Breaz} %\thanks{S. Breaz is supported by the Babe\c s-Bolyai University grant GSCE-30254}
\address{Babe\c s-Bolyai University, Faculty of Mathematics
and Computer Science, Str. Mihail Kog\u alniceanu 1, 400084
Cluj-Napoca, Romania}
\email{bodo@math.ubbcluj.ro}

\author{George Ciprian Modoi}
\address{Babe\c s-Bolyai University, Faculty of Mathematics
and Computer Science, Str. Mihail Kog\u alniceanu 1, 400084
Cluj-Napoca, Romania}
\email{cmodoi@math.ubbcluj.ro}

\date{\today}


\thanks{This research is supported by the grant UEFISCDI grant PN-II-ID-PCE-2012-3-0100.}

\begin{abstract}
We describe nil-clean companion matrices over fields. 
\end{abstract}


\maketitle

%\tableofcontents

\section{Introduction}

Nicholson proved in \cite{Nic-tams-77} that in the study of some important properties of rings it is important to know when we can decompose
some (or all) elements of a ring as a sum of an idempotent and an element with some special properties. For instance, an element $r$ of a 
ring $R$ is \begin{itemize} \item \textsl{clean} if $r=u+e$, where $e$ is idempotent and $u$ is a unit of $R$, \cite{Nic-tams-77}; 
\item \textsl{weakly-clean} if $r=u\pm e$, where $e$ is idempotent and $u$ is a unit of $R$, \cite{Ahn-An-06}; 
    
\item \textsl{nil-clean} if $r=n+e$, where $e$ is idempotent and $n$ is nilpotent, \cite{Diesl-13};

\item \textsl{weakly nil-clean} if $r=n\pm e$, where $e$ is idempotent and $n$ is nilpotent, \cite{Da-Mc-15}.
\end{itemize}
The classes of clean and nil-clean rings are closed with respect standard constructions, e.g. direct products of 
(nil-)clean rings
are (nil-)clean, and matrix rings over (commutative nil-)clean rings are 
(nil-)clean, \cite{HN} (resp. \cite{Br-et-al-13}). However, the classes of weakly (nil-)clean rings are not closed under 
these constructions (see \cite{Ahn-An-06} and \cite{Br-Da-Zh}).
Moreover, while all matrix rings over fields are clean, when we consider nil-clean rings there are strongly restrictions: 
if a matrix ring over a division ring $F$ is nil-clean then 
$F$ has to be isomorphic to $\mathbb{F}_2$, \cite{kosan-et-al-14}. 

It can be useful to know the (nil-)clean elements in some
rings which are not (nil-)clean. For instance, strongly clean matrices
(i.e. they have a decomposition $r=u+e$ such that $eu=ue$) are characterized in \cite[Theorem 3]{Nic-99} for general rings and in 
\cite[Theorem 37]{BDD} for matrices over commutative local rings are studied (we reffer to \cite{CYZ06} for a particular case). 
In particular it would be also nice to characterize nil-clean elements in matrix rings over division rings. 
For the case of strongly nil-clean elements (i.e. they have a decomposition $r=e+n$ such that $en=ne$) we refer to 
\cite[Theorem 4.4]{Diesl-13}.
From this result we conclude that an $n\times n$ matrix over a division ring $D$ is strongly nil-clean if and only if its characteristic polynomial has the form $X^{(n-m)}(X-1)^m$. For other studies of (nil-)clean elements in various rings we refer to 
\cite{Aro} and \cite{C14}.            

Since in the proofs of the fact that the matrix ring $\Mat_n(\bbF)$ over the field $\bbF$ 
(resp. $\bbF=\mathbb{F}_2$) is (nil-)clean the Frobenius 
(rational) normal form can be used, cf. \cite{Nic-pams-78}, \cite{Ya-Zh} and \cite{Br-et-al-13}, it is useful to know 
when a companion matrix is nil-clean. In the main result of the present paper (Theorem \ref{main-comp}) we characterize companion matrices over fields which are nil-clean. As a corollary of the main result we obtain a characterization for finite prime fields by the fact that every big enough companion matrix 
is nil-clean (Corollary \ref{prime-fields}). Moreover, it is proved 
that all these matrices have nil-clean like decompositions:
 for every polynomial $\chi$ of degree $n$ such that the coefficient of $X^{n-1}$ is 0, all { big enough companion} matrices can be decomposed as 
$A=E+B$ where $E$ is an idempotent and $B$ is a matrix whose characteristic polynomial is $\chi$. We can view this result as a partial answer to the general (open) problem, which is connected to the main result of \cite{camilo-02}, which asks to characterize the matrices
$A$ over  $\bbF$ which can be decomposed $A=B+C$ such that $\chi_B=f$ and $g(C)=0$, where
$f$ and $g$ are two fixed polynomials over a field $\bbF$. 
 In fact the problem 
of decompositions of a companion matrix in this way can be reduced to a matrix completion problem (for nil-clean decomposition 
we used Lemma \ref{prescr-char} and 
Proposition \ref{c=e+m}). We refer to \cite{Cr-09} for a nice survey on this subject.

In this paper, $\bbF$ will denote a field, $\Mat_n(\bbF)$ is the ring of all $n\times n$ matrices over $\bbF$, and $\bbF_p$ is the prime 
field of characteristic $p$ ($p$ is a prime). If $A\in\Mat_n(\bbF)$ we will denote by $\chi_A=\det(XI-A)$ the characteristic polynomial of $A$.

\section{Matrix completion results which involve idempotents}\label{section-comp}

Let $\mathbb{F}$ be a field and $n$ a positive integer. For every tuple $(c_0,c_1,\dots,c_{n-1})\in \mathbb{F}^n$, we denote by  
\[\tag{C} C=C_{c_0,c_1,\dots,c_{n-1}}=\left[\begin{array}{ccccc}
          0&0&\cdots&0&-c_0\\
          1&0&\cdots&0&-c_1\\
          0&1&\cdots&0&-c_2\\
          \cdots&\cdots&\cdots&\cdots&\cdots\\
          0&0&\cdots&0&-c_{n-2}\\
          0&0&\cdots&1&-c_{n-1}
         \end{array}\right]\] the $n\times n$  
companion matrix determined by $(c_0,c_1,\dots,c_{n-1})$. The characteristic polynomial associated to $C$ is denoted by 
$$\chi_{c_0,c_1,\dots,c_{n-1}}=X^n+c_{n-1}X^{n-1}+\cdots+c_1 X+c_0.$$

We are grateful to an anonymous referee for pointing out that the matrix in the following lemma not only have the same 
characteristic polynomial, but is even similar to the respective companion matrix:

\begin{lemma}\label{prescr-char} Let  $f=X^n+f_{n-1}X^{n-1}+\ldots+f_1X+f_0\in\bbF[X]$ be a monic polynomial. 
For every tuple  $(c_1,\dots,c_{n-1})\in \bbF^{n-1}$ there exists a unique tuple
$(\alpha_0,\alpha_1,\ldots,\alpha_{n-1})\in \bbF^n$ such that the matrix
\[M=\left[\begin{array}{ccccc}
          -\alpha_{n-1}&-\alpha_{n-2}&\cdots&-\alpha_1&-\alpha_0\\
          1&0&\cdots&0&-c_1\\
           0&1&\cdots&0&-c_2\\
          \cdots&\cdots&\cdots&\cdots&\cdots\\
          0&0&\cdots&0&-c_{n-2}\\
          0&0&\cdots&1&-c_{n-1}
         \end{array}\right]\in\Mat_{n}(\bbF)\] 
				is similar to the companion matrix $C_{f_0,f_1,\dots,f_{n-1}}$ of $f$.  
\end{lemma}

\begin{proof} 
First we claim that the $n$-tuple $(\alpha_0,\alpha_1,\ldots,\alpha_{n-1})\in\bbF^n$ is uniquely determined by the equality 
$\chi_M=f$.  
The characteristic polynomial of $M$ is 
\[\chi_M=\det(XI_n-M)=\left|\begin{array}{ccccc}
          X+\alpha_{n-1}&\alpha_{n-2}&\cdots&\alpha_1&\alpha_0\\
          -1&X&\cdots&0&c_1\\
          \cdots&\cdots&\cdots&\cdots&\cdots\\
          0&0&\cdots&X&c_{n-2}\\
          0&0&\cdots&-1&X+c_{n-1}
         \end{array}\right|.\]
Expanding it after the first line we obtain after some immediate computations:
				\begin{align*}\chi_M&=(X+\alpha_{n-1})\left|\begin{array}{cccc}
          X&\cdots&0&c_1\\
          \cdots&\cdots&\cdots&\cdots\\
          0&\cdots&X&c_{n-2}\\
          0&\cdots&-1&X+c_{n-1}
         \end{array}\right|\\ 
				&+\alpha_{n-2}\left|\begin{array}{cccc}
          X&\cdots&0&c_2\\
          \cdots&\cdots&\cdots&\cdots\\
          0&\cdots&X&c_{n-2}\\
          0&\cdots&-1&X+c_{n-1}
         \end{array}\right|+\ldots
				+\alpha_3\left|\begin{array}{ccc}
          X&0&c_{n-3}\\
          -1&X&c_{n-2}\\
					0&-1&X+c_{n-1}
         \end{array}\right|\\ 
				&+\alpha_2\left|\begin{array}{cc}
          X&c_{n-2}\\
          -1&X+c_{n-1}\\
         \end{array}\right|+\alpha_1(X+c_{n-1})+\alpha_0\\
				&=(X+\alpha_{n-1})\chi_{c_1,\dots,c_{n-1}}+\alpha_{n-2}\chi_{c_2,\dots,c_{n-1}}+\ldots+\alpha_1\chi_{c_{n-1}}+\alpha_0.
	\end{align*}
Therefore $\chi_M$ can be written as: 
\begin{align*}\chi_M&=X^n+(\alpha_{n-1}+c_{n-1})X^{n-1}+(\alpha_{n-2}+\alpha_{n-1}c_{n-1}+c_{n-2})X^{n-2}+\ldots \\
&+(\alpha_{n-i}+\alpha_{n-i+1}c_{n-1}+\ldots+\alpha_{n-1}c_{n-i+1}+c_{n-i})X^{n-i}+\ldots\\
&+(\alpha_1+\alpha_2c_{n-1}+\ldots+\alpha_{n-1}c_2+c_1)X
+(\alpha_0+\alpha_1c_{n-1}+\ldots+\alpha_{n-1}c_1+c_0),\end{align*}
where $c_0=0$ is added only for uniformity.  

The system \[(\alpha_{n-i}+\alpha_{n-i+1}c_{n-1}+\ldots+\alpha_{n-1}c_{n-i+1}+c_{n-i})=f_{n-i},\ { 1\leq i\leq n}\]
with the unknowns $\alpha_{0},\alpha_1,\ldots,\alpha_{n-1}$ has obviously a unique solution, proving the claim made at the 
very beginning of this proof. 

Finally we observe that the minimal polynomial of the matrix $M$ is equal to its characteristic polynomial. 
Indeed by direct computation we can see that { for every $1\leq k<n$ the} first column of the matrix 
$M^k$ has $1$ on the position $(k+1,1)$ and only 
$0$ bellow, that is on all positions $(i,1)$ with $k+1<i\leq n$. This implies the matrix obtained by evaluating at $M$ any monic 
polynomial of degree $k<n$ has an entry $1$ on the position $(k+1,1)$, hence it is not zero. Therefore the minimal 
polynomial of $M$ is of degree at least $n$, hence it is forced to coincide with $\chi_M$. 
In these conditions we only have to note the equality $\chi_M=f$ is enough in order to conclude that $M$ 
is similar to the companion matrix of $f$. \end{proof}

\begin{remark}\label{fixed-alpha}
 In Lemma \ref{prescr-char} above we may want $\alpha_{n-1},\ldots,\alpha_{n-i}$ to be fixed, $1\leq i\leq n$, and we can determine uniquely 
$f_{n-1},\ldots,f_{n-i}$ and $\alpha_{n-i-1},\ldots,\alpha_0$ such that $M$ is similar to the companion matrix of the polynomial $f$. 
In the following we will make use of a 
particular case of this remark, namely when $i=1$ and $\alpha_{n-1}$ is fixed. Noting that $f_{n-1}=\alpha_{n-1}+c_{n-1}$ we will
determine  
$\alpha_0,\alpha_1,\ldots,\alpha_{n-2}$ such that 
$\chi_M=X^n+(\alpha_{n-1}+c_{n-1})X^n+f_{n-2}X^{n-2}+\ldots+f_1X+f_0$, where $f_0,f_1,\ldots,f_{n-2}$ are arbitrary in $\bbF$.
\end{remark}

\begin{proposition}\label{c=e+m} Let $n=m+k$ be a positive integer, where $m,k\in\bbN^*$. 
Fix constants $c_0,c_1,\dots,c_{n-1}\in\bbF$ and denote as above $C=C_{c_0,c_1,\dots,c_{n-1}}$. For every polynomial 
$g\in\bbF[X]$ of degree at most $n-2$
there exist two matrices $E,M\in \Mat_{n}(\bbF)$ such that 
\begin{enumerate}[{\rm (1)}]
\item $C=E+M$;
\item $E$ is idempotent, and
\item $\chi_M=X^n+(k1+c_{n-1})X^{n-1}+g$.
\end{enumerate}
%  Consider also $f=X^n+(k1+c_{n-1})X^{n-1}+f_{n-2}X^{n-2}+\ldots+f_1X+f_0\in\bbF[X]$ be a polynomial of degree $n$ with the first two coefficients fixed to be $1$, respectively $(k1+c_{n-1})$, and the rest being arbitrary in $\bbF_p$. Then $C$ may be written as a sum $C=E+M$ with $E,M\in\Mat_{n\times n}(\bbF_p)$ such that $E=E^2$ is an idempotent and the characteristic polynomial of $M$ is 
%$p_M=f$.
\end{proposition}

\begin{proof}
Let $g=f_{n-2}X^{n-2}+\ldots+f_1X+f_0\in\bbF[X]$.
Consider the matrix $$E=\left[\begin{array}{cc}
          E_{11}&E_{12}\\
          E_{21}&E_{22}\\
         \end{array}\right]\in\Mat_n(\bbF),$$ where $E_{11}=I_k\in\Mat_{k}(\bbF)$, $E_{21}$ and $E_{22}$ are both $0$
(in $\Mat_{m\times k}(\bbF)$, respectively $\Mat_{m}(\bbF)$) and 
\[E_{12}=\left[\begin{array}{ccccc}
          0&0&\cdots&0&\alpha_0-c_0\\
          0&0&\cdots&0&\alpha_1-c_1\\
          \cdots&\cdots&\cdots&\cdots&\cdots\\
          0&0&\cdots&0&\alpha_{k-2}-c_{k-2}\\
          \alpha_{n-2}&\alpha_{n-3}&\cdots&\alpha_k&\alpha_{k-1}-c_{k-1}
         \end{array}\right]\] where $\alpha_0,\alpha_1,\ldots,\alpha_{n-2}\in\bbF$.
Then it may be immediately verified that $E=E^2$ is an idempotent. 
 We will prove by induction on $k\geq1$ that there are uniquely 
determined $\alpha_0,\alpha_1,\ldots,\alpha_{n-2}$ such that $M=C-E$ has the characteristic polynomial 
equal to $f=X^n+(k1+c_{n-1})X^{n-1}+g$. The step $k=1$ is just Lemma \ref{prescr-char} (with fixed $\alpha_{n-1}=1$ and $f_{n-1}=1+c_{n-1}$, 
as in Remark \ref{fixed-alpha}). 

Now suppose the claim is true for $k-1\geq1$, and let $M=C-E$ as before. Expanding after the first column we compute:
\begin{align*}
\chi_M&=\det(XI_n-M) \\ &=\left|\begin{array}{cccccccc}         
          X+1&0&\cdots&0&0&\cdots& 0 &\alpha_0\\ 
          -1&X+1&\cdots&0&0&\cdots& 0 &\alpha_1\\
          \cdots&\cdots&\cdots&\cdots&\cdots&\cdots& \cdots &\cdots\\
          0&0&\cdots&X+1&\alpha_{n-2}&\cdots& \alpha_{k-2} &\alpha_{k-1}\\
          0&0&\cdots&-1&X&\cdots & 0&c_k\\
          % 0&0&\cdots&0&-1&\cdots&c_{k+1}\\
          \cdots&\cdots&\cdots&\cdots&\cdots&\cdots&\cdots &\cdots\\
					0&0&\cdots&0&0&\cdots& X & c_{n-2}\\
          0&0&\cdots&0&0&\cdots& -1 & X+c_{n-1}
         \end{array}\right|\\
  &=(X+1)\left|\begin{array}{ccccccc}          
          X+1&\cdots&0&0&\cdots& 0 &\alpha_1\\
          \cdots&\cdots&\cdots&\cdots&\cdots&\cdots&\cdots\\
          0&\cdots&X+1&\alpha_{n-2}&\cdots& \alpha_{k-2}&\alpha_{k-1}\\
          0&\cdots&-1&X&\cdots& 0 &c_k\\
          % 0&\cdots&0&-1&\cdots&c_{k+1}\\
          \cdots&\cdots&\cdots&\cdots&\cdots&\cdots &\cdots\\
					0&\cdots&0&0&\cdots& X & c_{n-2}\\
          0&\cdots&0&0&\cdots&-1& X+c_{n-1}
         \end{array}\right|+\alpha_0.
\end{align*}  
Clearly 
the coefficient of $X^{n-1}$ in $\chi_M$ is $k1+c_{n-1}$.

By division with remainder theorem we obtain $f=(X+1)q+r$ with $\deg r\leq0$, and $\chi_M=f$ if and only if 
\[q=\left|\begin{array}{cccccc}          
          X+1&\cdots&0&0&\cdots&\alpha_1\\
          \cdots&\cdots&\cdots&\cdots&\cdots&\cdots\\
          0&\cdots&X+1&\alpha_{n-2}&\cdots&\alpha_{k-1}\\
          0&\cdots&-1&X&\cdots&c_k\\
          % 0&\cdots&0&-1&\cdots&c_{k+1}\\
          \cdots&\cdots&\cdots&\cdots&\cdots&\cdots\\
          0&\cdots&0&0&\cdots&X+c_{n-1}
         \end{array}\right|\]  and 
$r=\alpha_0$.  Then the determinant giving $q$ is the same form as $\det(XI_n-M)$ but of dimension 
$(n-1)\times(n-1)$ where $n-1=(k-1)+m$ and the coefficient of $X^{n-2}$ is $(k-1)1+c_{n-1}$.
We apply induction hypothesis in order to determine (uniquely) $\alpha_1,\cdots,\alpha_{n-2}$. 
\end{proof}

\section{Nil-clean companion matrices}

We will use the results proved in the previous section to give a complete description of nil-clean companion matrices
over fields. We start with 
a result which is valid for all matrices. Recall that an element of a ring is \textsl{unipotent} if it is a sum $1+b$ 
such that $b$ is nilpotent. 

\begin{lemma}\label{e=i}
Let $A\in\Mat_{n}(\bbF)$ be a (not necessarily companion) matrix which is nil-clean. The following are true:
\begin{enumerate}[{\rm (1)}]
\item There exists a non-negative integer $k$ such that $\tr A=k\cdot 1$. 
\item If $\mathrm{char}(\bbF)=0$ and $\tr A=k\cdot 1$ then  
\begin{enumerate}[{\rm (a)}]
\item $k\leq n$;
\item $k=0$ if and only if $A$ is nilpotent;
\item $k=n$ if and only if $A$ is unipotent.
\end{enumerate}
\item If $\mathrm{char}(\bbF)=p>0$ then 
\begin{enumerate}[{\rm (a)}] 
\item there exists $k\in\{1,\dots,p\}$ such that $\tr A=k\cdot 1$, and $k\leq n$ {or $k=p$};
\item { if $n<k=p$ then $A$ is nilpotent;}
\item if $k=n<p$ then $A$ is unipotent;
\item 
if $k=n= p$ then $A$ is unipotent or nilpotent.
\end{enumerate} 
\end{enumerate} 
\end{lemma}

\begin{proof} If $A=E+N$ with $E$ idempotent and $N$ nilpotent then $\tr A=\tr E+\tr N=\tr E$. Moreover it 
is known that if $E$ is an idempotent matrix then $\tr E=\mathrm{rank}(E)\cdot 1$ (see \cite[Lemma 19]{Br-Da-Zh}), 
so there is $k\in\N$, $k\leq n$ such that $\tr A=k\cdot 1$. This $k$ is unique if $\mathrm{char}(\bbF)=0$ and is unique only modulo 
$p$ if $\mathrm{char}(\bbF)=p$, so the statements (2)(a) and (3)(a) are obvious.   

(2)(b) If $k=0$ then $\rank E=0$, hence $E=0$, and this implies that $A$ is nilpotent. The converse is obvious.

(2)(c)  If $k=n$ then $\rank E=n$, hence $E=I_n$, and this implies that $A$ is unipotent. The converse is obvious.

{ (3)(b) Since $k=p$ we have $\tr A=0$, hence $\tr E=0$ and it follows that $p$ is a divisor for $\rank E$. But $\rank E<p$, and it follows 
that $\rank E=0$, hence $A$ is nilpotent. } 

(3)(c) If $k=n$ then $\rank E\equiv n(\mathrm{mod}\ p)$, and it is easy to see that in our hypothesis we have $\rank E=n$, hence $E=I_n$, and this implies that $A$ is unipotent.

(3)(d) If $A=E+N$ where $E=E^2$ is an idempotent and $N$ is nilpotent then $\rank E\cdot 1=0$, hence $\rank E\in\{0,p\}$. 
This implies $E\in\{0_n,I_n\}$.
Now the conclusion is obvious. 
\end{proof}

We are ready to enunciate the promised characterization for nil-clean companion matrices. Before this, let us remark that 
it is easy to verify if a (companion) matrix $C$ is nilpotent or unipotent since these matrices are characterized by the conditions  
$\chi_C=X^n$, respectively $\chi_C=(X-I_n)^n$.

\begin{theorem}\label{main-comp}
Let $C=C_{c_0,c_1,\dots,c_{n-1}}\in\Mat_{n}(\bbF)$ be a companion matrix. The following are equivalent:
\begin{enumerate}[{\rm (1)}]
\item $C$ is nil-clean.
\item One of the following conditions is true:
\begin{enumerate}[{\rm (i)}]
\item $C$ is nilpotent (i.e. $c_0=\dots=c_{n-1}=0$);
\item $C$ is unipotent (i.e. $c_i=(-1)^{i}\binom{n}{n-i}$ for all $i\in\{0,\dots,n-1\}$);
\item $\mathrm{char}(\bbF)=0$ and there exists a positive integer $k$ such that $-c_{n-1}=k\cdot 1$ and $n>k>0$;
\item $\mathrm{char}(\bbF)=p\neq 0$ and there exists an integer $k\in \{1,\dots,p\}$ such that $-c_{n-1}=k\cdot 1$ and $n>k$.
\end{enumerate}
\end{enumerate}
\end{theorem}
%There exists a minimal non-negative integer $k$ such that $\tr A=k1$ and one of the following conditions are satisfied:
\begin{proof}
(1)$\Rightarrow$(2) Suppose that $C$ is nil-clean, but it is not nilpotent nor unipotent. 

If $\mathrm{char}(\bbF)=0$ then we can use Lemma \ref{e=i} to observe that there exists a unique non-negative integer $k\leq n$ such 
that $-c_{n-1}=k\cdot 1$. Since $C$ is not nilpotent nor unipotent we obtain $0<k<n$. 

If $\mathrm{char}(\bbF)=p>0$ then we can use again Lemma \ref{e=i} to observe that there exists a unique integer $k\in\{1,\dots,p\}$ such 
that $-c_{n-1}=k\cdot 1$. If $0\neq E\neq I_n$ is an idempotent such that $C-E$ is a nilpotent matrix 
then $\rank E\equiv k\ (\mathrm{mod}\ p)$. { Since $\rank E\neq 0$ we have 
$n\geq \rank E\geq k$.} But $E\neq I_n$, hence $\rank E\neq n$, and we obtain $n>k$.   

(2)$\Rightarrow$(1) If $C$ is nilpotent or { unipotent}, then it is obviously nil-clean. Moreover, if we are in one of the casses (iii) or (iv),
we can apply Proposition \ref{c=e+m} for $g=0$ to obtain a nil-clean decomposition for $C$.
\end{proof}

\begin{remark}
(a) The conditions (i)--(iv) stated in Theorem \ref{main-comp} are not independent: if $n>p$ then the nilpotent $n\times n$ companion matrix
verify (iv).

(b) In general the nil-clean decompositions obtained using Proposition \ref{c=e+m} are not the unique nil-clean decompositions for 
companion matrices. For instance, if $c_{n-1}=1$ then $C_{c_0,\ldots,c_{n-1}}-C_{0,\ldots,0}$ is idempotent.       
\end{remark}

\begin{example}
(a) If $n=2$ and $p$ is a prime then the companion matrix $C=C_{c_0,c_1}$ over $\bbF_p$ is nil-clean, 
if and only if $c_0=c_1=0$ ($C$ is nilpotent) or $c_0=1$ and $c_1=-2$ ($C$ is unipotent) or $c_{1}=-1$ (in this case 
$C=\left[\begin{array}{cc} 0 & -c_0 \\ 0 & 1 \end{array}\right]+\left[\begin{array}{cc} 0 & 0 \\ 1 & 0 \end{array}\right]$).

(b) If $n=3$ and $C=C_{c_0,c_1,c_2}$ is a companion matrix over $\bbF_p$ we have the following cases:
\begin{itemize}
\item if $c_2=0$, $C$ is nil-clean exactly in one of the following situations: 
\begin{enumerate}[{\rm (i)}] 
\item $p=2$,
\item $p=3$ and $C$ is unipotent or nilpotent,  
\item $p\geq 5$ and $C$ is nilpotent;
\end{enumerate}

\item if $c_2=-1$ then $C$ is nil-clean since 
$N=\left[\begin{array}{ccc}
-1 & c_1-1 & 2c_1-1\\
1 & 0 & -c_1 \\
0 & 1 & 1
\end{array}\right]$
is nilpotent and $C-N$ is idempotent; 

\item if $c_2=-2$ then $C$ is nil-clean since 
$N=\left[\begin{array}{ccc}
-1 & 0 & 1\\
1 & -1 & p-3 \\
0 & 1 & 2
\end{array}\right]$
is nilpotent and $C-N$ is idempotent;

\item if $p\geq 5$ and $c_2=-3$ then $C$ is  nil-clean if and only if it is unipotent;

\item if $p\geq 5$ and $c_2=-k$ with $k\in \{4,\dots,p-1\}$ then $C$ is not nil-clean.
\end{itemize} 

(c) Finally, we illustrate the decompositions obtained by using Proposition \ref{c=e+m} for the case $n=4$ and $p=3$:
\begin{itemize}
\item for $c_3=0$ we obtain that 
$$N=\left[\begin{array}{cccc} -1 & 0 & 0 & -1\\
										1 & -1 & 0 & 1 \\
										0 & 1 & -1 & 0 \\
										0 & 0 & 1 & 0
\end{array}\right]$$
is nilpotent and $C_{c_0,c_1,c_2,0}-N$ is idempotent;

\item for $c_3=1$ we obtain that 
$$N=\left[\begin{array}{cccc} -1 & c_2-1 & c_1-c_2-1 & -c_1-c_2^2-1\\
										1 & 0 & 0 & -c_1 \\
										0 & 1 & 0 & -c_2 \\
										0 & 0 & 1 & 1
\end{array}\right]$$
is nilpotent and $C_{c_0,c_1,c_2,1}-N$ is idempotent;

\item for $c_3=-1$ we obtain that 
$$N=\left[\begin{array}{cccc} -1 & 0 & 0 & -1\\
										1 & -1 & c_2 & 1 \\
										0 & 1 & 0 & -c_2 \\
										0 & 0 & 1 & -1
\end{array}\right]$$
is nilpotent and $C_{c_0,c_1,c_2,0}-N$ is idempotent. 
\end{itemize} 
 
\end{example}

As a corollary of this theorem we obtain a generalization for \cite[Theorem 1]{Br-et-al-13}:

\begin{corollary}\label{prime-fields}
Let $n\geq 3$ be a positive integer. The following are equivalent for a field $\bbF$:
\begin{enumerate}[{\rm (1)}]
\item $\bbF\cong \bbF_p$ for a prime  $p<n$; 
\item every companion matrix $C\in \Mat_{n}(\bbF)$ is nil-clean;
\item if  $C\in \Mat_{n}(\bbF)$ is a companion matrix then for every polynomial $g\in\bbF[X]$ of degree at most $n-2$
there exist two matrices $E,M\in \Mat_{n\times n}(\bbF)$ such that $C=E+M$,
$E$ is idempotent, and
$\chi_M=X^n+g$.
\end{enumerate}
\end{corollary}

\begin{proof}
(1)$\Rightarrow$(3) For every companion matrix $C$ we can write $\tr C=k\cdot 1$ with $k\in\{1,\dots,p\}$, and we can apply Proposition \ref{c=e+m}.

(3)$\Rightarrow$(2) This is obvious.

(2)$\Rightarrow$(1) Since every element of the field $\bbF$ can be the trace of a companion matrix, it follows from Lemma \ref{e=i} that 
every element from $\bbF$ has the form $k\cdot 1$, $k\in\mathbb{N}$. This implies that there exists a prime $p$ such that $\bbF\cong \bbF_p$.
Moreover, if we supose $p\geq n$, then use can use Theorem \ref{main-comp} to observe that the companion matrix $C=C_{(-1)^{n+1},0,\ldots,0}$ is not nil-clean (it is not nilpotent nor unipotent, and it does not verify the condition (iv) from Theorem \ref{main-comp}). Therefore $p<n$, and the proof is complete.  
\end{proof}

\begin{remark}
It is not hard to see that if 
all companion matrices which appear in the Frobenius normal form of a matrix $A$ are nil-clean then $A$ is also nil-clean.
It would be nice to know if the converse of this remark is also true. For a proof of a very particular case we refer to 
\cite{Br-Da-Zh}, where it is proved that for $p=3$ and every $n\geq 2$ the matrix 
$A_n=\mathrm{diag}(1,-1,0,\dots,0)\in \Mat_n(\bbF_3)$ 
is not nil-clean. Note that the Frobenius normal form for the matrix $A_4$ is 
$\mathrm{diag}(0;C_{0,-1,0})$, where $0\in\Mat_1(\bbF_3)$ 
is obviously nil-clean and 
$C_{0,-1,0}\in \Mat_3(\bbF_3)$ is not nil-clean. 

On the other side, let us remark that if we consider only diagonals of companion 
matrices (non necessarily Frobenius normal forms) it is possible to obtain nil-clean matrices 
even the companion matrices are not nil-clean. 
For instance it is not hard to see that the minimal polynomial of the matrix 
$A=\mathrm{diag}(C_{0,0};C_{-1,0})\in \Mat_{4}(\bbF_3)$ is of degree $4$. Therefore $A$ is 
nil-clean since it is similar to a $4\times 4$ companion matrix over $\bbF_3$. 
On the other side $C_{-1,0}\in \Mat_2(\bbF_3)$ is not nil-clean.
\end{remark}

\noindent \textbf{Acknowledgements.} This study has the starting point in the work of \c Stefana Sorea, \cite{Stef}. 
She proved by some direct computations that all $4\times 4$ companion matrices over $\bbF_3$ are nil-clean.

We would like to thank to the referee for his/her critical remarks, which helped us to improve the paper.

 \end{document}